\pdfoutput=1

\documentclass [11pt]{article}

\usepackage{amsmath,amsthm,enumitem,multicol,amssymb,mathrsfs,verbatim,graphicx,subcaption,authblk}
\usepackage{bbm}
\usepackage{geometry,xcolor}
\usepackage[makeroom]{cancel}
\usepackage[normalem]{ulem}
\makeatletter
\def\uwave{\bgroup \markoverwith{\lower3.5\p@\hbox{\sixly \textcolor{red}{\char58}}}\ULon}
\font\sixly=lasy6
\makeatother

\DeclareMathSymbol{\R}{\mathbin}{AMSb}{"52}

\renewcommand{\R}{\mathbb{R}}

\DeclareMathSymbol{\N}{\mathbin}{AMSb}{"4E}
\DeclareMathSymbol{\Z}{\mathbin}{AMSb}{"5A}

\newtheorem{lemma}{Lemma}

\newtheorem{theorem}{Theorem}

\begin{document}

 \title{On a Generalized Fibonacci Recurrence}
 
\author{Natasha Blitvi\'{c}}
\affil{Department of Mathematics and Statistics\\ Lancaster University}

\author{Vicente I. Fernandez}

\affil{Department of Civil, Environmental and Geomatic Engineering\\ ETH Z\"urich}

\maketitle

\begin{abstract}
The generalized Fibonacci recurrence
$g_n=g_{n-k}+g_{n-m}$
was recently used to demonstrate the theoretically optimal nature of limited senescence in morphologically symmetrically dividing bacteria. Here, we study this recurrence from a more abstract viewpoint, as a general model for asymmetric branching, and interpret solutions for different initial conditions in terms of branching-related quantities.  We provide a compact diagrammatic representation for the evolution of this process which leads to an explicit binomial identity for the sums of elements lying on the diagonals $kx+my=n$ in Pascal's triangle $\N_0\times \mathbb N_0\ni(x,y)\mapsto {x+y\choose x}$, previously sought by Dickinson~\cite{Dickinson1950}, Raab~\cite{Raab1963}, and Green~\cite{Green1968}.
\end{abstract}

\section{Introduction}

Motivated by an active debate in the microbiology community concerning whether senescence occurs in unstressed morphologically symmetrically dividing bacteria such as \emph{Escherichia coli} \cite{stewart2005aging, wang2010robust, rang2012ageing} and questions over the possible benefits of two very similar growth equilibria \cite{proenca2018age} achieved through asymmetric partitioning at division, we recently utilized in \cite{BlitvicFernandez} the linear recurrence
\begin{equation}g_n=g_{n-k}+g_{n-m}\tag{GF}\label{eq-GF}\end{equation}
as a model for asymmetric microbial division.  Senescence in morphologically symmetrically dividing bacteria is of particular interest to the evolution of ageing \cite{ackermann2007evolutionary, chao2010model}, where the development of physical asymmetries between parent and child organisms is an important step.  We showed that \eqref{eq-GF} can be interpreted in terms of a generalized Bellman-Harris process with non-identically distributed offspring (accounting for the division asymmetry) in scenarios where the variance of the distributions governing the two lifetimes is vanishingly small.
Expressing the asymptotics of the solutions of \eqref{eq-GF} in terms of the unique positive root of the characteristic equation
$1-x^k-x^m=0$
yielded the corresponding growth rate of bacterial populations exhibiting bimodal features.
  In \cite{BlitvicFernandez}, we used this model to examine the benefits of small growth asymmetries when dividing and identified conditions under which the observed small ($\leq5$\% \cite{lapinska2019bacterial, proenca2018age, stewart2005aging}) asymmetries in \emph{Escherichia coli} would be optimal, providing a meaningful fitness advantage over symmetric division.  

At present, we take a step back and study \eqref{eq-GF} as a \emph{generalized Fibonacci recurrence}.
Over the years, various generalizations of Fibonacci numbers have been proposed and investigated, including, notably, Lucas' 1878 study of the recurrence $g_n=ag_{n-1}+bg_{n-2}$ (for fixed $a,b\in\mathbb N$ and initial conditions $g_0,g_1\in\mathbb N_0$) ~\cite{Lucas1878}. More generally, given some some fixed coefficients $\alpha_1,\ldots,\alpha_\ell$ and initial values $u_0,\ldots,u_{\ell-1}$, the sequence $(u_n)$ satisfying \begin{equation}
u_n=\alpha_1 u_{n-1}+\alpha_2 u_{n-2}+\ldots+\alpha_\ell u_{n-\ell}
\label{eq-gen-rec}\end{equation}
is equivalently determined (see e.g. \cite{Levesque1985}) by the generating function   \begin{equation}
\sum_{n=0}^\infty u_n x^n=\frac{u_0+u_1x+\ldots+u_{\ell-1}x^{\ell-1}}{1-\alpha_1x-\alpha_2x^2-\ldots-\alpha_\ell x^\ell}.
\end{equation}
Standard analysis methods yield an explicit formula for $(u_n)$ in terms of the roots of the characteristic polynomial $1-\alpha_1x-\alpha_2x^2-\ldots-\alpha_\ell x^\ell$, which gives a generalization of the \emph{Binet formula} for Fibonacci numbers. Further algebraic properties of this general class of sequences and of their associated objects (such as companion matrices) have been studied (see e.g. \cite{Dubeau1997,Chen1996,BenTaher2003,Schork2007,Yang2008}).

While some unifying insight can be obtained by treating generalized Fibonacci numbers as a special case of the general recurrence \eqref{eq-gen-rec},  different specializations come with distinct combinatorial and geometric insights, as evidenced by the extensive study of the \emph{Lucas numbers} ($l_0=2$, $l_1=1$, $ l_n=l_{n-1}+l_{n-2}$),  \emph{Pell numbers} ($p_0=0$, $p_1=1$, $ p_n=2p_{n-1}+p_{n-2}$), and  \emph{Padovan numbers}  ($q_0=q_1=q_2=1$, $q_n=q_{n-2}+q_{n-3}$).
Further specializations include the case  $\alpha_{\ell-1}=\alpha_\ell=1$ of \cite{Conolly1981} and 
$\alpha_1=\ldots=\alpha_\ell=1$ of \cite{Miles1960} (see also \cite{Goyt2013} for a more combinatorial perspective), which includes the so-called tribonacci ($\ell=3$) and quadranacci ($\ell=4$) sequences. 

The generalized Fibonacci recurrence \eqref{eq-GF}, considered at present, was previously introduced in \cite{Dickinson1950} (see also \cite{Raab1963,Green1968}), motivated by the combinatorial problem of characterizing the sums of terms appearing on diagonals of Pascal's triangle. (More recently, \cite{Northshield2011} studied the `continuous version' of the same problem.) For special values of $k$ and $m$ and specific initial values, solutions of \eqref{eq-GF} have previously appeared in various combinatorial and algebraic settings (see e.g. the references for the entry
A005686 in \cite{oeis} for the case $k=2,m=5$). The case $k=1$ has previously attracted attention in computer science \cite{Peterson1977} and biology, especially in the context of asymmetric cell division in yeast \cite{Spears1998, Olofsson2011}. 

In this work, we develop further mathematical properties of \eqref{eq-GF} as a general model for asymmetric branching. We give a more concrete combinatorial interpretation of the sequences arising from \eqref{eq-GF}, for $k,m\in\mathbb N$ (assumed to be relatively prime), under different initial conditions, and assign to these a convenient diagrammatic representation in terms of quantities appearing in a deformed version of Pascal's triangle (the \emph{$(k,m)$-Pascal's triangle}).   Exploiting the relations between these quantities, we obtain an explicit expression for the sums of terms on the diagonals of classical Pascal's triangle, given as a sum of binomial coefficients without relying on integer solutions to the Diophantine equation $kx+my=n$. This provides a missing piece in the earlier work of \cite{Dickinson1950}, Raab \cite{Raab1963} and Green \cite{Green1968}. 


\section{Generalized Fibonacci numbers as a model for asymmetric branching}

Consider the model where a bacterium is born at time $0$ and undergoes its first division at time $m$. The time from birth until division is the \emph{lifetime} of the bacterium. Starting with this division and for all divisions thereafter, each bacterium in the system splits into a ``faster-growing" child that will undergo division at time $k$ and ``slower-growing" one, which will undergo division at time $m>k$. (See Figure~\ref{fig-tree-Pascal} (a).) This model accurately describes the bimodal morphologically stable asymmetric cell division (see \cite{BlitvicFernandez} for some examples) and is a generalization of analogous models for the division of yeast \cite{Olofsson2011} (case $k=1$). Beyond biology, the model is more generally applicable to any situation that exhibits binary branching with persistent asymmetry.




Denote by $a_n^{(k,m)}$, for $n=0,1,2,\ldots$, the total number of bacteria at time $t\in [n,(n+1))$, by $b^{(k,m)}_n$ and $c^{(k,m)}_n$ the number of those bacteria with lifetimes $k$ and $m$, respectively, and by $d_n^{(k,m)}$  the number of simultaneous divisions occurring at any time $n$. In \cite{BlitvicFernandez}, we pointed out that, analogously to the relationship between the Fibonacci and Lucas numbers, the aforementioned biologically motivated sequences all satisfy \eqref{eq-GF}, starting from different initial values.  Namely:

\begin{theorem}[Supplemental Materials to \cite{BlitvicFernandez}]
The following sequences are solutions to \eqref{eq-GF}:
\begin{enumerate}[noitemsep]
 \item $(a_n^{(k,m)})_n$, the number of bacteria alive in the time interval $[n,(n+1))$, subject to the initial condition $a_0^{(k,m)}=1=\ldots=a_m^{(k,m)}$;
 \item $(d_n^{(k,m)})_n$, the number of bacterial divisions occurring at time $n$, with the initial condition $d_n^{(k,m)}=0=\ldots=d_{m-1}^{(k,m)},d_{m}^{(k,m)}=1$;
 \item $(b_n^{(k,m)})_n$ the number of bacteria of lifetime $k$, alive in the time interval $[n,(n+1))$, with the initial condition $b^{(k,m)}_0=\ldots=b^{(k,m)}_{m-1}=0$ and $b^{(k,m)}_{m}=\ldots=b^{(k,m)}_{m+k-1}=1$;
 \item $(c_n^{(k,m)})_n$, the number of bacteria of lifetime $m$, alive in the time interval $[n,(n+1))$, with the initial condition $c^{(k,m)}_0=\ldots=c^{(k,m)}_{m+k-1}=1$.\footnote{Of the four sequences mentioned, $(a_n^{(k,m)})$ and $(d_n^{(k,m)})$ coincide when $k=1$, but are distinct whenever $k>1$. In general, $a_n^{(k,m)}=b_n^{(k,m)}+c_n^{(k,m)}$, as the latter two reflect the composition of the bacterial population. However, as the underlying combinatorial model is based on a representation of  Pascal's triangle, both $(b_n^{(k,m)})$ and $(c_n^{(k,m)})$ work out to be time-shifted versions of the sequence $(a_n^{(k,m)})$. } 
 \end{enumerate}
 \label{thm-1}
 \end{theorem}

 Hence, several solutions to the abstract two-term recurrence \eqref{eq-GF} correspond to concrete combinatorial quantities, all of which are captured in the unifying framework of asymmetric branching. Furthermore, from Theorem~\ref{thm-1}, the generating function, Binet-type formula, and large-$n$ asymptotic expression can be straightforwardly computed for all four sequences. (See the Supplementary Materials to \cite{BlitvicFernandez}.) In particular, when $k$ and $m$ are relatively prime, we have that for large $n$, 
\begin{equation}a_n^{(k,m)}\sim \frac{(1-r_0^k)}{(1-r_0)(mr_0^{m}+kr_0^{k})}\,r_0^{-n},\label{eq-asymptotic}\end{equation}
where $r_0$ is the unique positive root of the characteristic equation $1-x^k-x^m=0$.
Furthermore, the fractions of the population corresponding to each subtype converge, as time gets large, to
\begin{equation}\lim_{n\to\infty}\frac{b_n^{(k,m)}}{a_n^{(k,m)}}=1-r_0^k=r_0^m\quad\quad\text{and}\quad\quad \lim_{n\to\infty}\frac{c_n^{(k,m)}}{a_n^{(k,m)}}=r_0^k.\label{eq-proprtions}\end{equation}

\section{The $(k,m)$-Pascal's triangle}

Rather than depicting asymmetric branching as a binary tree, as in Figure~\ref{fig-tree-Pascal}(a), we introduce a more compact representation by combining objects within a generation that have the same remaining lifetime. The result is Pascal's triangle that appears `skewed' compared to the standard representation. 
Specifically, the \emph{$(k,m)$-Pascal's triangle} is a collection of weighted, horizontal edges on the vertex set $\N_0\times \N_0$, where for each $x,y=0,1,2,\ldots$, the edge $[yk+xm,yk+(x+1)m)\times \{y\}$ has weight ${y+x-1\choose x-1}$. (See Figures~\ref{fig-tree-Pascal} and \ref{fig-ad}.) The equivalence between the two representations is established in the following: 

\begin{figure}\centering\includegraphics[scale=0.4]{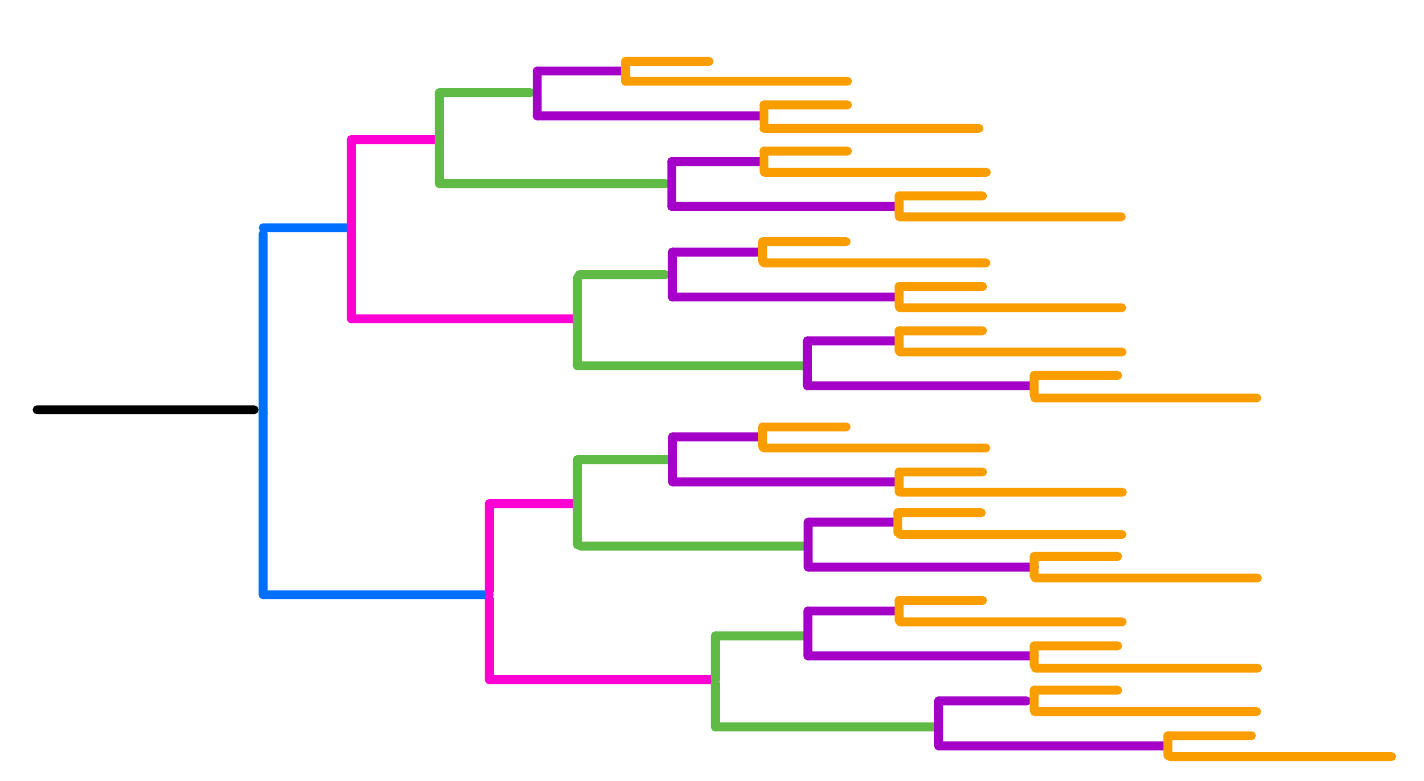}\includegraphics[scale=0.7]{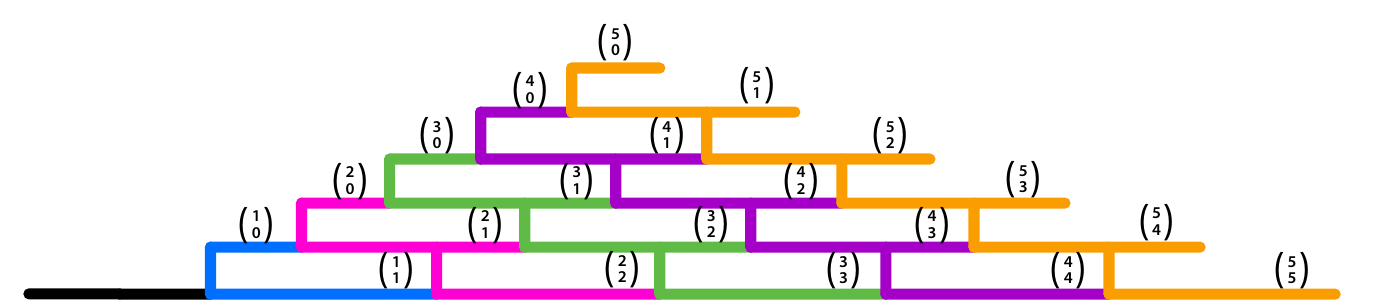}\centering\caption{ Five generations of bacterial division for $k=2,m=5$: (a) as a binary tree, (b) as the $(2,5)$-Pascal's triangle. Each color denotes a different generation}\label{fig-tree-Pascal}\end{figure}

\begin{lemma}
Fix $k,m\in\mathbb N$ with $k<m$. For all $n\geq m-k$, the number of bacteria alive at time $[n,(n+1))$, denoted by $a_n^{(k,m)}$, equals the sum of the weights of the edges intersecting the ordinate $x=n+m-k$ in the $(k,m)$-Pascal's triangle. 
\label{lemma-a-weights}
\end{lemma}
\begin{proof}

Consider all the bacteria belonging to the $i$th generation, for $i\in\mathbb N$, in a $(k,m)$-heterogeneous system, recalling that the $0$th generation consists of a single mother bacterium of fixed lifetime $m$. By considering the possible lifetimes of each of its $2^{i-1}$ predecessors, we observe that each bacterium in the $i$th generation must have been born at one of the following times: $m+ik,2m+(i-1)k, 3m+(i-2)k,\ldots,im+k,(i+1)m$. Furthermore, there are $2{i\choose 0}$ bacteria born at time $m+ik$, $2{i\choose 1}$ bacteria born at time $2m+(i-1)k$, $2{i\choose 2}$ bacteria born at time $3m+(i-2)k$, and generally $2{i\choose \alpha}$ bacteria born at time $(1+\alpha)m+k,(i-\alpha)m$, where half of the bacteria born at each time have length $k$ and half have length $m$.

Now, map the mother bacterium of generation zero to the edge $[0,k)\times\{0\}$ of multiplicity one. Thereafter, for each bacterium in the $i$th generation that has lifetime $k$ and is born at time $\alpha m+(i-\alpha)k$ (for $\alpha=0,\ldots,i$), map it to a distinct (multi-)edge $[\alpha m+(i-\alpha)k,\alpha m+(i-\alpha+1)k)\times\{i-\alpha\}$. In contrast, for each bacterium in the $i$th generation born at time $\alpha m+(i-\alpha)k$ and of lifetime $m$, map it to the pair of (multi-)edges $[\alpha m+(i-\alpha)k,(2+\alpha)m+(i-\alpha-1)k)\times\{i-\alpha-1\}$ and $[(2+\alpha)m+(i-\alpha-1)k,(2+\alpha)m+(i-\alpha)k)\times\{i-\alpha-1\}$. Then, each edge $[\alpha m+(i-\alpha)k,(2+\alpha)m+(i-\alpha-1)k)\times\{i-\alpha-1\}$ has multiplicity ${i\choose \alpha}$, as it's in the image under the previously described map of all bacteria of lifetime $m$ born at time  $\alpha m+(i-\alpha)k$. Moreover, each edge $[\alpha m+(i-\alpha)k,\alpha m+(i-\alpha+1)k)\times\{i-\alpha\}$ has multiplicity 
\begin{equation}{i\choose \alpha}+{i\choose \alpha -1}={i+1 \choose \alpha }\label{eq-pascal-rule}\end{equation}
as it may be the image of a bacterium of lifetime $k$ born at time $\alpha m+(i-\alpha)k$ or that of a bacterium of lifetime $m$ born at time $(2+\alpha)m+(i-\alpha-1)k$. Note that \eqref{eq-pascal-rule} is the well known \emph{Pascal's rule} ${n+1\choose k}={n\choose k}+{n\choose k-1}$, which holds for all integers $1\leq k\leq n+1$.

The bacterium born at time $0$ has lifetime $m$, but is represented by a line segment of length $k$. Barring the $0$th generation, each bacterium's lifetime is otherwise represented in the resulting diagram by a line segment of the corresponding length and the multiplicity of the edges corresponds to the number of bacteria represented by each edge. Hence, the number of bacteria alive at time $n$ is the number of edges (counted with multiplicities) intersected by the line $x=n-m+k$.

The $(k,m)$-Pascal's triangle is now obtained by turning the multi-edges into weighted edges and combining contiguous edges that have equal weights. Namely, in the $0$th generation, combine the edge $[0,k)\times\{0\}$ (multiplicity $1$, generation $0$) with the edge $[k,m)$ (multiplicity $1$, generation $1$) to form the edge $[0,m)$ of weight $1$. Thereafter, for each $i\in\mathbb N$ and $\alpha=0,\ldots,i$, map the multi-edges
$[\alpha m+(i-\alpha)k,\alpha m+(i-\alpha+1)k)\times\{i-\alpha\}$ (multiplicity ${i+1 \choose \alpha }$, generation $i$) and 
$[\alpha m+(i-\alpha+1)k,(2+\alpha)m+(i-\alpha)k)\times\{i-\alpha\}$ (multiplicity ${i+1\choose \alpha}$, generation $i+1$) to the single edge $[\alpha m+(i-\alpha)k,(2+\alpha)m+(i-\alpha)k)\times\{i-\alpha\}$ of weight ${i+1\choose \alpha}$. The reader may now establish that each edge in the $(k,m)$-Pascal's triangle is uniquely recovered in this manner. Moreover, each edge thus produced carries the correct weight.
\end{proof}

\section{The sum of diagonals in (classical) Pascal's triangle}

It is well known that the $n$th Fibonacci number gives the sum of the elements of Pascal's triangle $(x,y)\mapsto {x+y\choose x}$ that lie on the diagonal $x+2y=n$ (see Figure~\ref{fig-pascal-diagonals}). This is the combinatorial interpretation of the identity
\begin{equation}f_n=\sum_{j=0}^{\lfloor\frac{ n-1}{2}\rfloor}{n-j-1\choose j},\label{eq-fib-binomial}\end{equation}
where $f_n$ is the $n$th Fibonacci number. In \cite{Green1968}, Green  generalized \eqref{eq-fib-binomial} to 
\begin{equation}T_n^{(1,m)} =\sum_{j=0}^{\lfloor\frac{n}{m}\rfloor}{n-(m-1)j\choose j},\label{eq-Green}\end{equation}
where  $m$ is a fixed integer greater than 1 and $T_n^{(1,m)}$ is the analogous sum over the diagonal $x+my=n$ ($n\in\mathbb N$). More generally, for the sums $T_n^{(k,m)}$ over the diagonals $kx+my=n$ (for fixed $k,m\in\mathbb N, k\neq m$), an expression of the form \eqref{eq-fib-binomial} can be written down for each of the values of the remainder of the division $n/k$ (or $n/m$). This is essentially worked out in Section E of \cite{Green1968}, though the result is somewhat cumbersome. A subset of these diagonals is characterized in \cite{Dickinson1950}.\footnote{The reader should note that Dickinson considers a different representation of Pascal's triangle in $\mathbb N_0\times \mathbb N_0$, namely, $(x,y)\mapsto {x\choose y}$, rather than $(x,y)\mapsto {x+y\choose y}$, considered by Green. When transformed into the framework of Green, Dickinson's diagonals are of the form $y=(-k/m)x+nk/m+\frac{m}{m-k}$. Compared to the case of interest, that is, $y=(-k/m)x+n/m$, Dickinson's $y$-intercepts are offset by $\frac{m}{m-k}$ and, more to the point, are spaced apart by multiples of $k/m$ rather than $1/m$.}

Drawing on the the results of the previous two sections, we obtain a general form of \eqref{eq-Green}, namely, a binomial identity for the sequence $(T_n^{(k,m)})_n$ in the case of $k$ and $m$ general. 
\begin{theorem}
 Fix $k,m\in\mathbb N$ with $\gcd(k,m)=1$. For all $n=0,1,2,\ldots,$
   \begin{equation}
  T_n^{(k,m)}=\sum_{j=0}^{\lfloor (n+m)/k\rfloor} {\lfloor \frac{n-jk}{m}\rfloor+j+1\choose j+1}-\sum_{j=0}^{\lfloor (n+m-1)/k\rfloor} {\lfloor \frac{n-1-jk}{m}\rfloor+j+1\choose j+1}.\label{eq-diagonals-general-2}
 \end{equation}
 \end{theorem}
 \begin{proof}
By construction,
\begin{equation}T_n^{(k,m)}=\sum_{\substack{x,y\in\{0,1,2,\ldots\}\\ \text{s.t. } kx+my=n}}{{x+y}\choose x}.\label{eq-T}\end{equation}
Green \cite{Green1968} showed that, for fixed $k,m\in\mathbb N$ with $\gcd(k,m)=1$,
$T_n^{(k,m)}=T_{n-k}^{(k,m)}+T_{n-m}^{(k,m)}$ for $n\geq m$.
By inspection, for $n=0,\ldots,m-1$ 
\begin{equation}T_n^{(k,m)}=\mathbbm{1}_{k|n}\quad\text{ and }\quad T_0^{(k,m)}=1,\end{equation}
where $\mathbbm{1}_{k|n}=1$ when $k$ divides $n$ and $\mathbbm{1}_{k|n}=0$ otherwise. (For a worked example for $k=2,m=5$, see Figure 7 of \cite{fibonacci-Pascal}.)  By Theorem~\ref{thm-1}, the  $(T_n^{(k,m)})$ is a shifted version of the sequence $(d_n^{(k,m)})$. Indeed, write $m=\alpha k+\beta$ where $\alpha$ and $\beta$ are positive integers (the quotient and the remainder of the division). Then, by comparing the values of $T_n^{(k,m)}$ for $0\leq n\leq m$ to the values of $d_n^{(k,m)}$ for $m\leq n\leq 2m$, we see that $d_n^{(k,m)}=T_{n-m}^{(k,m)}$ for all $m\leq n\leq 2m$. (Specifically, for $n\leq m$, the sequence $(T_n^{(k,m)})$ has $1$'s in positions $0,k,2k,\ldots,\alpha k, m$, and zeros otherwise. For $m\leq n\leq 2m$, the sequence $(d_n^{(k,m)})$ has $1$'s in positions $m,m+k,m+2k,\ldots, m+\alpha k,2m$, and zeros otherwise.) Since the sequences satisfy the same recurrence, 
we obtain that for all $n\geq m$,
\begin{equation}d_n^{(k,m)}=T_{n-m}^{(k,m)}.\label{eq-T-d}\end{equation}
By Theorem~\ref{thm-1}, $(d_n^{(k,m)})_n$ gives the consecutive differences of the sequence $(a_n^{(k,m)})_n$, that is,
\begin{equation}a_n^{(k,m)}=1+\sum_{i=1}^n d_i^{(k,m)}.\label{eq-a-sum-d}\end{equation}
 In turn, by Lemma~\ref{lemma-a-weights}, the sequence $(a_n^{(k,m)})_n$ sums over vertical slices of the $(k,m)$-Pascal's triangle. Combining \eqref{eq-T-d} and \eqref{eq-a-sum-d}, we obtain that for all $n\geq m$, \begin{equation}\sum_{i=m}^n T_{i-m}^{(k,m)}=-1+\sum_{j=0}^{\lfloor n/k\rfloor} {\lfloor \frac{n-jk}{m}\rfloor+j\choose j+1}.\label{eq-diagonals-general}\end{equation}
The expression \eqref{eq-diagonals-general-2} now follows by taking successive differences.
\end{proof} 
 
Expressions of the form \eqref{eq-diagonals-general-2} or \eqref{eq-diagonals-general} do not appear to have previously been found beyond the special case $k=1$. In particular, \eqref{eq-diagonals-general-2} is considerably more explicit than Green's main result (see (1), E(a.3) and E(a.4) in \cite{Green1968}), which hinges on computing the minimal and maximal non-negative integer solutions of the equation $kx+my=n$, and is more general than that of \cite{Dickinson1950}, which considers a subset of the diagonals.

\section{Conclusion}

In an area where the problems are treated in either great generality, as in \eqref{eq-gen-rec}, so that the combinatorial and geometric distinctions between various specializations are lost, or in such specificity that different parameter choices and different initial conditions in \eqref{eq-GF} give rise to entirely different combinatorial problems, the results presented here offer a satisfyingly balanced perspective. Juxtaposing the interpretation of \eqref{eq-GF} in terms of asymmetric branching with the representation of the underlying sequences in terms of Pascal's triangle yields a new identity for sums over arbitrary diagonals in Pascal's triangle, extending beyond the previously studied $k=1$ special case. As a result, though the original motivation for this model was for bacterial growth, it is likely to have considerably broader applications, such as in operations research \cite{Conolly1981} and computer science \cite{Peterson1977}.\\




\noindent {\bf Acknowledgment}\, The authors are grateful to Roman Stocker and Einar Steingr{\'i}msson for helpful conversations and to Martin Ackermann for his hospitality during the first author's sabbatical leave.

\begin{figure}
\centering
\begin{subfigure}{.6\textwidth}
  \centering
\includegraphics[width=\linewidth]{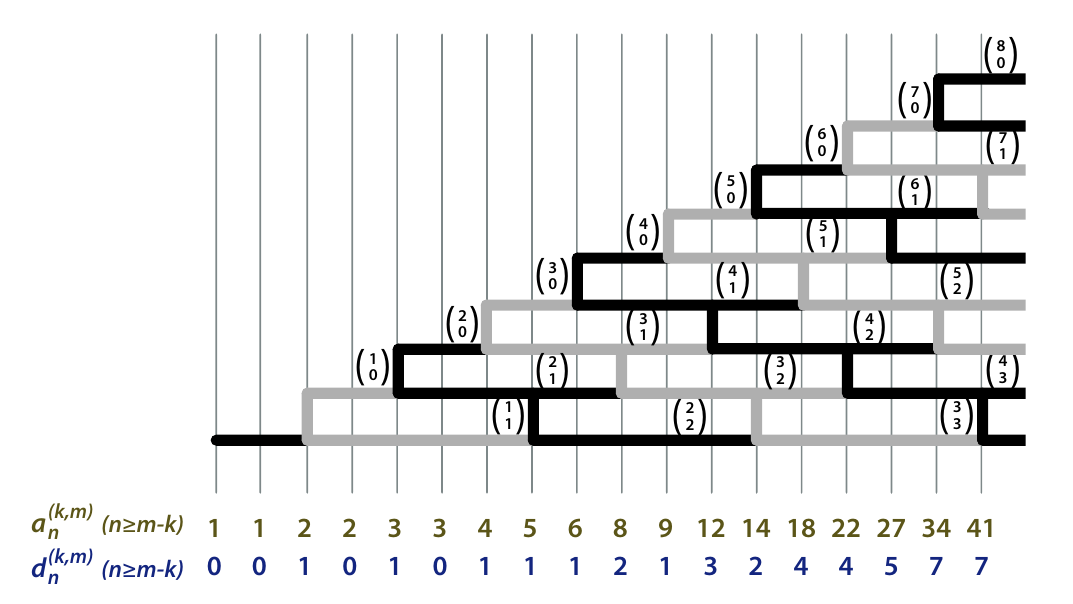}
  \caption{The sequences $(a_n^{(k,m)})$ and $(d_n^{(k,m)})$  for $k=2,m=5$.}
\label{fig-ad}
\end{subfigure}%
\begin{subfigure}{.4\textwidth}
  \centering
  \includegraphics[width=\linewidth]{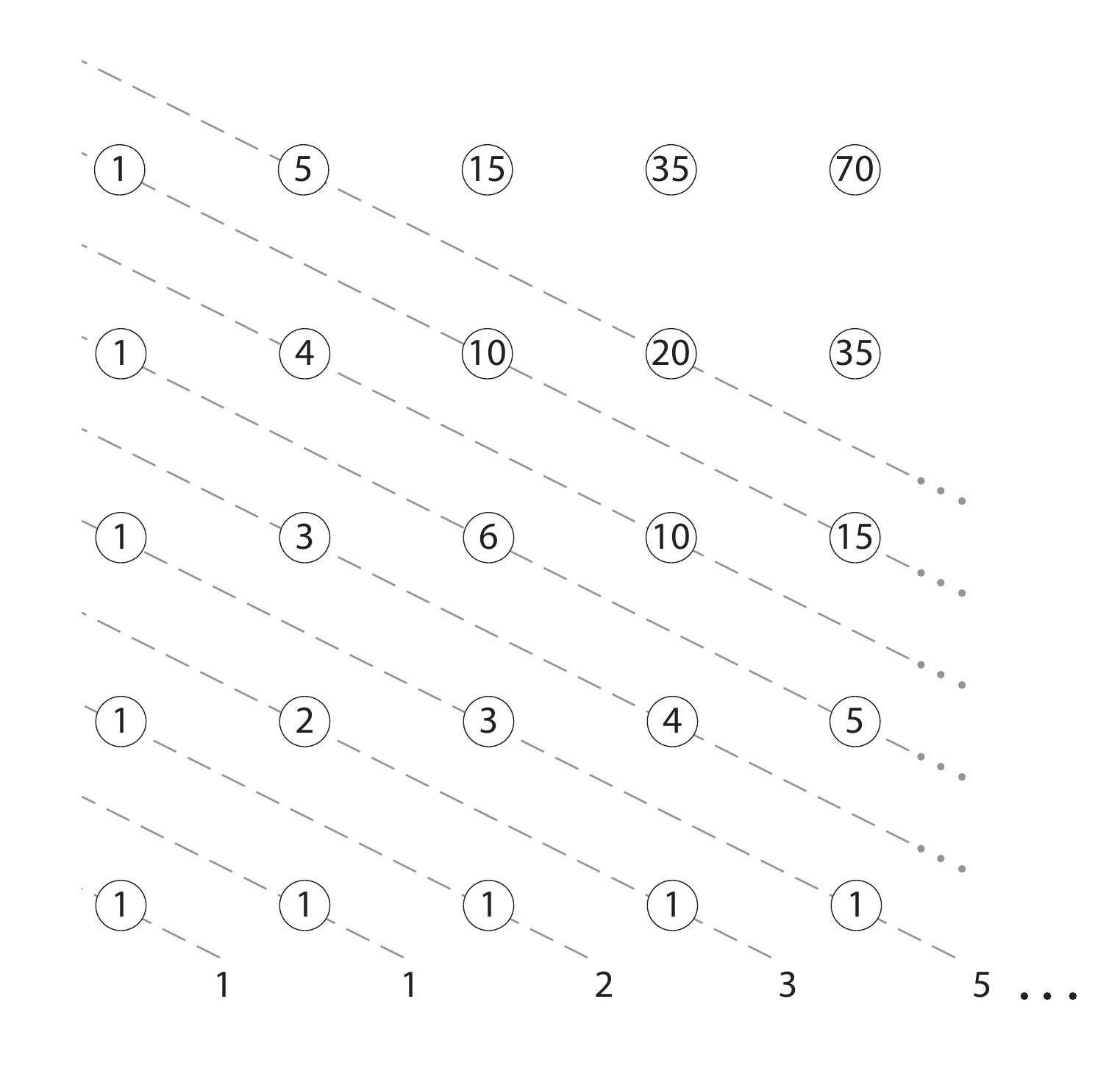}
  \caption{$(d^{(k,m)})$ as sums over diagonals $kx+my=n$ in Pascal's triangle. The case $k=1,m=2$ gives the Fibonacci numbers.}
\label{fig-pascal-diagonals}
\end{subfigure}
\caption{}
\label{fig:test}
\end{figure}

\bibliographystyle{alpha}
\bibliography{phenotypic_refs2} 
\end{document}